\documentclass[12pt,twoside]{article}
\usepackage{latexsym,amssymb,amsthm,amsmath,lscape,epsfig,setspace,tikz,slashed,mathtools}
\usepackage[all]{xy}
\usepackage[retainorgcmds]{IEEEtrantools}
\usetikzlibrary{decorations.markings}
 \usetikzlibrary{arrows}
\usepackage[pdftex]{hyperref}
\usepackage[margin=10pt,font=small,labelfont=bf]{caption}
\usepackage{mathrsfs}

\usepackage{combelow}

\theoremstyle{plain}
\newtheorem{theorem}{Theorem}[section]
\newtheorem{lemma}[theorem]{Lemma}

\newtheoremstyle{namedthm}{}{}{\itshape}{}{\bfseries}{\!\!.}{0.5em}{\thmnote{#3 }}
\theoremstyle{namedthm}

\theoremstyle{definition}
\newtheorem{definition}[theorem]{Definition}

\newtheorem{example}[theorem]{Example}

\newenvironment{renumerate}%
{%
\begin{enumerate}}%
{\end{enumerate}%
}%

{\vskip6pt%
\noindent%
{\it Remark.}}%
{\vskip6pt}

{\vskip6pt%
\noindent%
{\it Remarks}. %
\begin{renumerate}}%
{\end{renumerate}\vskip6pt}

\makeatletter
\def\Ddots{\mathinner{\mkern1mu\raise\p@
\vbox{\kern7\p@\hbox{.}}\mkern2mu
\raise4\p@\hbox{.}\mkern2mu\raise7\p@\hbox{.}\mkern1mu}}
\makeatother

\newcommand{\R}{\text{${\mathbb R}$}}
\newcommand{\C}{\text{$\mathbb C$}}
\newcommand{\NN}{\text{$\mathbb N$}}

\renewcommand{\frak}[1]{\text{$\mathfrak{#1}$}}
\renewcommand{\tilde}{\widetilde}

\renewcommand{\bar}{\overline}

\newcommand{\tensor}{\otimes}

\newcommand{\mc}[1]{\text{$\mathcal{#1}$}}

\newcommand{\noqed}{\let\qed\relax}

\newcommand{\gcs}{generalized complex structure}

\newcommand{\gcss}{generalized complex structures}

\newcommand{\gc}{generalized complex}

\newcommand{\gcy}{generalized Calabi--Yau}
\newcommand{\gcym}{generalized Calabi--Yau manifold}
\newcommand{\gcys}{generalized Calabi--Yau structure}

\date{} \usepackage{color} \definecolor{tocolor}{rgb}{.1,.1,.5}
\definecolor{urlcolor}{rgb}{.2,.2,.6}
\definecolor{linkcolor}{rgb}{.1,.1,.6}
\definecolor{citecolor}{rgb}{.6,.2,.1}
\hypersetup{backref=true, colorlinks=true, urlcolor=urlcolor,
  linkcolor=linkcolor, citecolor=citecolor}
\definecolor{darkgreen}{rgb}{0.0, 0.5, 0.0}

\begingroup
\hypersetup{linkcolor=blue}
\endgroup

\numberwithin{equation}{section}

\begin{document}

\title{Type changing, compact generalized Calabi--Yaus}
\author{Gil R. Cavalcanti}
\maketitle

\abstract{We provide the first compact examples of type-changing generalized Calabi--Yau manifolds using  Fei's construction of non-K\"ahler Calabi--Yau manifolds on branched covers of twistor spaces \cite{Fei16}. The structure, which is generically of type 1 and changes to type 3, is constructed on the product of a hyperelliptic surface of genus three and a four-torus, $\Sigma_3 \times \mathbb{T}^4$, with starting point a non-K\"ahler Calabi--Yau structure.
}
\vskip12pt
\noindent
MSC classification 2020: 	53D18, 81Q60. \\
Subject classification: Differential geometry.

\setcounter{tocdepth}{1}


\section{Introduction}

Generalized complex structures were introduced by Hitchin \cite{Hit03} and further developed by Gualtieri \cite{Gua07} in 2003. They caught the attention of mathematicians and physicists for being a simultaneous generalization of complex and symplectic structures. In fact, one can continuously deform a complex structure into a symplectic one through generalized complex structures \cite{Hit03}. Further, generalized complex structures are natural geometric structures underlying string theory, supersymmetry and mirror symmetry \cite{GHR84,Gua14,GMPT04,LMTZ05}.

At each point, generalized complex structures split into two disconnected classes, of even and
odd type. Within each class, the space of generalized complex structures is stratified according
to the type. In real dimension $2n$, type $0$ and type $1$ form the generic strata in the even
and odd cases, respectively, while types $n-1$ and $n$ form the lowest-dimensional strata.
Type $0$ corresponds to symplectic structures, whereas type $n$ corresponds to complex
structures. All intermediate types are genuinely new to the theory.

An important feature is that the type need not be constant: it is an upper semicontinuous
function that may vary across the manifold. Early examples showed that symplectic and complex
points can coexist on the same connected manifold \cite{Gua07}. This phenomenon has subsequently led to the construction of many examples of generalized complex $4$-manifolds that are neither complex nor symplectic \cite{CG09,Tor12,TY14,CKW20}.

A distinguished class of generalized complex structures consists of those with holomorphically
trivial canonical bundle. Hitchin called these \emph{generalized Calabi--Yau manifolds} \cite{Hit03}, since the condition agrees with the usual Calabi--Yau condition for a K\"ahler
manifold. Generalized Calabi--Yau structures are also relevant to string theory: this condition
is precisely what permits Witten's topological twist \cite{GMPT04,LMTZ05}.

This raises a natural question: {\it can a compact generalized Calabi--Yau manifold exhibit type
change?}

Examples in $\mathbb{R}^{2n}$ are easy to construct, so the difficulty lies in achieving
compactness. The constructions that have produced many type-changing generalized complex
structures in dimension four do not yield generalized Calabi--Yau structures, and no
alternative, surgery-based, construction has been produced. Another possible approach is
to deform a complex Calabi--Yau manifold. For such a deformation, one needs a closed
form $\alpha\in\Omega^{1,0}(M),$
whose zero locus becomes the type-change locus. If $M$ is K\"ahler, however, then $M$ admits a
Calabi--Yau metric, and a closed $(1,0)$-form is parallel. It must therefore either vanish
identically or be nowhere zero. This suggests that non-K\"ahler geometry is essential to the
construction.

In this paper we produce the first compact examples of type-changing generalized Calabi--Yau
manifolds. The construction follows the deformation approach above and uses twistor spaces,
hyperelliptic curves, and ideas introduced by Fei \cite{Fei16} for constructing non-K\"ahler
Calabi--Yau manifolds.

\section{Preliminaries}

In this section we introduce \gc\ and \gcy\ structures, give a few examples, outline the difficulties  and setup the eventual solution of the problem. 

\begin{definition} A \emph{generalized complex structure} on a $2n$-dimensional  manifold $M$ is a line subbundle $K \subset  \wedge^\bullet T^*_{\C} M$ such that pointwise any generator $\rho$ of $K$ is of the form
\begin{equation}\label{eq:alg condition 1}
\rho = e^{B+i\omega} \wedge \Omega
\end{equation}
where $B,\omega \in \Omega^2(M;\R)$, $\Omega$ is a decomposable $k$-form, for some $k \in \NN$ and
\begin{equation}\label{eq:alg condition 2}
\omega^{n-k}\wedge \Omega\wedge\bar{\Omega} \neq 0.
\end{equation}
Further, if $\rho$ is a nonvanishing local section of $K$, then there are a local vector field, $X$, and 1-form, $\xi$, such that
\[d\rho = \iota_X \rho + \xi\wedge \rho.\]
The \emph{type} of the generalized complex structure at a point $p$ is the degree of the form $\Omega$ at that point.
\end{definition}

Among  \gcss, a special class consists of those determined by globally defined closed forms:
\begin{definition}
A \emph{generalized Calabi--Yau} structure is a generalized complex structure for which the line bundle $K$ admits a nowhere vanishing closed section.
\end{definition}

It is good to have a few examples and non-examples in mind.
\begin{example}
A complex structure is a particular example of a generalized complex structure in which we take $K$ to be the canonical bundle, $\wedge^{n,0}T^*M$. The generalized Calabi--Yau condition is that the canonical bundle should be holomorphically trivial.   In complex geometry, including the non-K\"ahler context, manifolds satisfying this property are often referred to as complex Calabi--Yau manifolds.
\end{example}

\begin{example}
A symplectic structure, $\omega \in \Omega^2(M)$  induces a \gcys\ determined by the closed form $e^{i\omega}$.
\end{example}

\begin{example}[Noncompact type-changing \gcy]
For any triple of constants $\alpha$, $\beta$, $\gamma$,  the following form defines a generalized Calabi--Yau structure on $\C^3$:
\begin{equation}\label{eq:3d unimodular}
\rho = \alpha z_1dz_1 + \beta z_2dz_2 +\gamma z_3dz_3 + dz_1\wedge dz_2 \wedge dz_3.
\end{equation}

This structure has type 3 at the origin and, depending on the values of $\alpha$, $\beta$ and $\gamma$,  possibly at coordinate lines and planes. For $(\alpha,\beta,\gamma) \neq 0$, it has type 1 in the complement of the coordinate planes. 

More generally, given a complex Lie algebra, $\frak{g}$, the associated linear holomorphic Poisson structure  on $\frak{g}^*$, $\pi \in \wedge^2\frak{g}^*\tensor \frak{g}$ gives rise to a generalized complex structure
\[\rho = e^\pi\cdot \Omega,\]
where $\Omega$ is a complex volume form for $\frak{g}^*$. Then $\rho$ is a generalized Calabi--Yau structure on the real vector space underlying $\frak{g}^*$  if and only if $\frak{g}$ is unimodular \cite{ELW99}. If $\frak{g}$ is not Abelian, this \gcs\ has type change. The 3-parameter family in \eqref{eq:3d unimodular}  corresponds to all possible unimodular complex 3-dimensional Lie algebras.
\end{example}

The following elementary observation tells us where not to look for type-changing \gcys.

\begin{lemma}\label{lem:not type zero}
If a connected \gcym\ has type 0 at some point, then it is globally of type 0.
\end{lemma}
\begin{proof}
Indeed, writing the defining form
\[\rho = \rho_0 + \rho_2 +\cdots,\]
where $\rho_i$ has degree $i$, we see that the structure has type 0 at some point if and only if $\rho_0$ is not identically zero. Then the generalized Calabi--Yau condition implies that $d\rho_0 =0$, which makes $\rho_0$ locally constant. Since $M$ is connected, $\rho_0$ is a constant and nonvanishing function, hence the type is zero at all points.
\end{proof}

\section{Type-changing compact generalized Calabi--Yaus}

By Lemma \ref{lem:not type zero}, a type-changing generalized Calabi--Yau manifold cannot have type 0 at any point. We therefore consider the next possibility: a structure that is generically of
type $1$ and reaches type $3$ along a nonempty locus in a 6-dimensional manifold.

We try to produce such an example by a deformation argument. Suppose $M$ is a complex Calabi--Yau threefold with holomorphic volume form
\[
\Omega\in\Omega^{3,0}(M).
\]
If we can find a closed $(1,0)$-form
\[
\alpha\in\Omega^{1,0}(M)
\]
whose zero locus is nonempty, then
\[
\rho=\alpha+\Omega
\]
defines a generalized Calabi--Yau structure. Indeed, $\rho$ is closed, and it satisfies the
algebraic conditions (2.1) and (2.2). Its type is $1$ wherever $\alpha\neq0$ and $3$ precisely
where $\alpha=0$.

On a K\"ahler Calabi--Yau manifold this is impossible: with respect to a Calabi--Yau metric,
every closed $(1,0)$-form is parallel. Hence it is either identically zero or nowhere vanishing.
We must therefore look outside the K\"ahler setting.

To produce our non-K\"ahler example, we rely on Fei's construction of non-K\"ahler Calabi--Yau manifolds using twistor spaces \cite{Fei16}.

\begin{example}[Type-changing \gcy]
Let $Z$ be the twistor space of the flat Riemannian 4-torus, $\mathbb{T}^4 = \R^4/\Gamma$, for a co-compact lattice $\Gamma$. That is, as a differentiable manifold $Z = \C P^1 \times \mathbb{T}^4$, where $\C P^1$ is the space of linear complex structures on $\R^4$ compatible with the flat metric. The complex structure on $Z$ is the product of the standard complex structure on $\C P^1$ and, for each $x \in \C P^1$, the tautologial structure on $\{x\}\times\mathbb{T}^4$. That is, the point  $x\in \C P^1$ determines the complex structure on $\mathbb{T}^4$. Since the torus is flat, this structure is integrable \cite{AHS78}.

By construction, as a complex manifold $Z$ is not the product of $\C P^1$ and $\mathbb{T}^4$ and the projection $Z \to \mathbb{T}^4$ is smooth, but not holomorphic. Notice however that the projection onto the first factor
\[\pi \colon Z \to \C P^1\]
is a holomorphic map.

Next, let $p\colon\Sigma \to \C P^1$ be a hyperelliptic curve of genus 3, for example, the one given by the polynomial
\[y^2 = x^8-1,\]
which is a double cover of $\C P^1$ with branch points the $8^{th}$-roots of unit.

The canonical bundle of a hyperelliptic surface of genus $g$ is given by $p^*\mc{O}(g-1)$, so in our case, $K_\Sigma = p^*\mc{O}(2)$.

We can then form the pull-back bundle
\[\xymatrix{ \tilde{Z} = p^*Z  \ar[r]\ar[d]^{\tilde\pi}&  Z\ar[d]^{\pi}\\
\Sigma \ar[r]^p& \C P^1}\]
Then the canonical bundle of $p^*Z$ is  \cite{Fei16}[Theorem 4.6, Example 4.8]
\[K_{\tilde{Z}} = \pi^*K_\Sigma \tensor p^* \circ \pi^*\mc{O}(-2) = \pi^* \circ p^*(\mc{O}(2)\tensor \mc{O}(-2)) = \mc{O}.\]
Therefore the manifold $\tilde{Z}$, which is diffeomorphic but not biholomorphic to $\Sigma \times T^4$, has trivial canonical bundle.

Since the  maps $p$ and $\pi$ are holomorphic, so is $\tilde\pi$, the pullback of ${\pi}$. Hence, given any holomorphic form, $\alpha \in \Omega^{1,0}(\Sigma)$, $\tilde\alpha = \tilde{\pi}^*\alpha$ is a holomorphic $(1,0)$-form on $\tilde{Z}$. Since $\Sigma$ has genus 3, $\alpha$ has generically four zeros and hence $\tilde\alpha$ vanishes on four fibers of $\tilde{Z}\stackrel{\tilde{\pi}}{\to}\Sigma$.

The closed form $\tilde\alpha + \Omega$ is a type-changing Calabi--Yau structure on $\tilde{Z}$, where $\Omega$ is an arbitrary holomorphic volume form on $\tilde{Z}$.
\end{example}

\section*{Acknowledgments}
I (and others) have known for a long time that the path used above would lead to an example of a type-changing Calabi--Yau manifold. The problem was the lack of examples of non-K\"ahler Calabi--Yau manifolds with the desired properties. Recent advances in large language models helped identify Fei’s construction and reproduce the relevant argument.

At the human side, I thank Gueo Grantcharov for helping with a quick double check of plausibility of the argument.

\bibliographystyle{hyperamsplain-nodash}
\bibliography{references}

{{
  \bigskip
  \footnotesize

  G. R.~Cavalcanti\par\nopagebreak \textsc{Department of Mathematics, Universiteit Utrecht, The Netherlands.}\par\nopagebreak
  \textit{E-mail}: \texttt{gil.cavalcanti@gmail.com.}

%
 }}

\end{document}